\documentclass[12pt]{article}

\usepackage{amsfonts}
\usepackage{amssymb}
\usepackage{amsthm}
\usepackage{amsmath}
\usepackage{mathrsfs}
\usepackage{stmaryrd}
\usepackage[latin1]{inputenc}
\usepackage[all]{xy}
\usepackage[left=1.20in,right=1.20in]{geometry}

\newcommand{\ent}{\mathbb Z}

\theoremstyle{plain}
\newtheorem{teo}{Theorem}

\newtheorem{coro}[teo]{Corollary}
\newtheorem{lema}[teo]{Lemma}

\theoremstyle{definition}
\newtheorem{defi}[teo]{Definition}
\newtheorem{nota}[teo]{Notation}

\theoremstyle{remark}

\newtheorem{rem}[teo]{Remark}

\title{Cyclic cellularity and active sums}

\author{Nadia Romero\footnote{Departamento de Matem\'aticas, Universidad de Guanajuato. Guanajuato, Mexico.}}

\date{ }

\begin{document}
\maketitle

\begin{abstract}
Let $G$ be a group and  let $\mathcal{F}$ be a family of subgroups of $G$ closed under conjugation. For a positive integer $n$, let $C_n$ denote a cyclic group of order $n$. We show that if there exists an integer $n$ such that every group in $\mathcal{F}$ is $C_n$-cellular  and has finite exponent diving $n$, then the active sum $S$ of $\mathcal{F}$ is $C_n$-cellular. We  obtain a couple of interesting consequences of this result, using results about cellularity. Finally, we give different proofs of the facts that Coxeter groups are $C_2$-cellular and that many groups of the form $\mathrm{SL}(n,\,q)$ for $n\geq3$ are $C_3$-cellular.
\end{abstract}

\section*{Introduction}

The \textit{group theoretical cellularization} of a group $G$ was developed by Rodr\'iguez and Scherer in \cite{rodsc} as an analogue in the category of groups of the cellularization of spaces. In recent years there have been important developments in the subject, as can be seen from \cite{flor2}, \cite{farj} and \cite{chac}, as well as from other references in the introduction of \cite{chac}. On the other hand, the notion of \textit{active sum} appeared in a paper of Tomás \cite{tomasana} as a generalization of the direct sum of groups, but this time taking into account the mutual actions of the groups in question. In its present form, the active sum of an \textit{active family} of subgroups of a group $G$ can be defined as a certain colimit in the category of groups (see Section 1.1. in \cite{sumasact} for details). Proving that a given group is the active sum of a  family of subgroups is not an easy task, but many examples have been considered in \cite{sumasact}, \cite{sumasact2} and \cite{ours}, dealing in particular with the question of  when a given group can be recovered as the active sum of a family of cyclic subgroups.

It was during a talk  about active sums at the EPFL, that J\'er\^ome Scherer observed that the active sum of a family of subgroups of $G$ seemed to share some nice properties with a \textit{cellular cover} of $G$ (compare for example Theorem 1 in \cite{chac} with the definition of active sum, or Lemma 1.5 in \cite{farj} with Lemma 1.5 in \cite{sumasact}). We will see that being the active sum of a family of cyclic subgroups is in general a stronger condition than being cellular for a cyclic group. In Theorem \ref{sumcel}, we prove cellularity with respect to a cyclic group for the active sum $S$ of a family of subgroups of $G$, subject to certain conditions. Using some results about cellularity we obtain two consequences of this, the first one is about the primes dividing the Schur multiplier of $S$ and the second one regards the question of when an $A$-cellular group, with $A$ cyclic, is (isomorphic to) the active sum of a family of cyclic subgroups. As a final consequence, we obtain a couple of examples of groups which are $A$-cellular for a cyclic group $A$.

\section{Definitions and notation}


For the active sum, we take the definition given in Section 1.2 of \cite{sumasact}, but we consider only families with the order given by equality. In this setting, the definition can be given as in Section 2.1 of \cite{ours}, that is:


\begin{defi}
Let $\mathcal{F}$ be a family of distinct subgroups of $G$  closed under conjugation ($\forall F\in\mathcal{F}, g\in G: F^g=g^{-1}Fg\in \mathcal{F}$). The \textit{active sum} $S$ of $\mathcal{F}$ is the free product of the elements of $\mathcal{F}$ divided by the normal subgroup generated by the elements of the form $h^{-1}\cdot g\cdot h\cdot (g^h)^{-1}$, with $h\in F_1$, $g\in F_2$,  $F_1, F_2\in\mathcal{F}$ (and thus, $g^h\in F_2^h=h^{-1}F_2h\in\mathcal{F}$).
\end{defi}

We note that if the family is generating ($\langle\bigcup_{F\in\mathcal{F}}F\rangle=G$), there is a surjective homomorphism $\varphi: S\to G$; see Section 1.2 of \cite{sumasact}.

Observe that if $G$ is a finite group, then the active sum of any family of distinct subgroups of $G$, closed under conjugation, is finite too. 

\begin{nota}
The letters $G$, $X$ and $Y$ will denote groups. For a positive integer $n$, we will write: 
\begin{itemize}
\item[i)] $C_n$ for a multiplicative cyclic group of order $n$.
\item[ii)] $G_n$ for $\{x\in G\mid x^n=1\}$.
\item[iii)] $\pi(n)$  for the set of primes dividing $n$. If $G$ is a finite group, then $\pi(G)$ will stand for $\pi(|G|)$. 
\end{itemize}

The Schur multiplier of $G$, the group $H_2(G,\, \ent)$, will be denoted by $H_2(G)$.
\end{nota}

We will take as our definition of an $A$-cellular group the one given in Definition 2.2 of \cite{chac}.

\begin{defi}[Definitions 2.1 and 2.2 in \cite{chac}]
Let $A$ be a group. A group homomorphism $f:X\rightarrow Y$ is called an $A$\textit{-equivalence}  if the map
\begin{displaymath}
\mathrm{Hom}(A,\, f):\mathrm{Hom}(A,\, X)\rightarrow \mathrm{Hom}(A,\, Y)
\end{displaymath}
induced by composition with $f$ is a  bijection. The homomorphism $f$ is called an $A$\textit{-injection}, if $\mathrm{Hom}(A,\, f)$ is an injection, and it is called $A$\textit{-trivial} if the image of $\mathrm{Hom}(A,\, f)$ consists of only the trivial homomorphism $1_{A,\, Y}:A\rightarrow Y$, $a\mapsto 1$.

A group $G$ is called $A$\textit{-cellular} if every $A$-equivalence is also a $G$-equivalence; it is called $A$\textit{-generated} if every $A$-trivial homomorphism is also $G$-trivial, and it is called $A$\textit{-constructible} if for every group $T$, the condition $\mathrm{Hom}(A,\, T)=\{1_{A,\, T}\}$ implies $\mathrm{Hom}(G,\, T)=\{1_{G,\, T}\}$. 
\end{defi}


\section{$C_n$-cellularity}

\begin{lema}\label{inje}
Let $m$ and $n$ be positive integers such that $m$ divides $n$.
\begin{itemize}
\item[a)] Every $C_n$-equivalence $f: X\rightarrow Y$ induces a bijection between $X_m$ and~$Y_m$.
\item[b)] Every $C_m$-cellular group is also $C_n$-cellular.
\end{itemize}
\end{lema}
\begin{proof}
a) Suppose that $f:X\rightarrow Y$ is a $C_n$-equivalence. Let $x$ and $x'$ be two elements of $X_m$ such that $f(x)=f(x')$. If $C_n$ is generated by $g$, then there are homomorphisms $h_1$ and $h_2$ from $C_n$ to $X$ satisfying $h_1(g)=x$ and $h_2(g)= x'$. The previous equality implies $fh_1=fh_2$, which implies $h_1=h_2$ and so $x=x'$. Now, given an element $y\in Y_m$ we can define a homomorphism $t:C_n\rightarrow Y$ which sends $g$ to $y$. But then there exists a homomorphism $h:C_n\rightarrow X$ such that $t=fh$. By taking $x_0=h(g)$, we have that $f(x_0)=y$. Finally, $y^m=1$ implies $f(x_0^m)=f(1)$, but clearly $x_0^m$ is in $X_n$, since $h(g)$ is, and $f$ is injective on this set, so we must have $x_0^m=1$.

b) Suppose that $G$ is a $C_m$-cellular group and let $f:X\rightarrow Y$ be a $C_n$-equivalence. We will show that $f$ is a $C_m$-equivalence to obtain the result.

Let $t_1$ and $t_2$ be two homomorphisms from $C_m$ to $X$ and suppose $ft_1=ft_2$. Since the images of $t_1$ and $t_2$ are contained in $X_m$, and by a) $f$ is injective on this set, we have that $t_1=t_2$. Now let $h$ be a homomorphism from $C_m$ to $Y$ and suppose $C_m$ is generated by~$g$. Since the image of $h$ is contained in $Y_m$, there exists and element $x\in X_m$ such that $f(x)=h(g)$. But then we can define $t':C_m\rightarrow X$ by sending $g$ to $x$ and we have that $ft'=h$.

\end{proof}

\begin{teo}
\label{sumcel}
Suppose $\mathcal{F}$ is a family consisting of distinct subgroups of $G$ of finite exponent and 
closed under conjugation. If there exists a positive integer $n$ such that for every $F\in \mathcal{F}$ the exponent of $F$ divides $n$ and $F$ is $C_n$-cellular, then the active sum $S$ of the family $\mathcal{F}$ is $C_n$-cellular.
\end{teo}
\begin{proof}
As explained in Section 1, the active sum in this case is the quotient of the free product $\amalg_{F\in \mathcal{F}} F$ by the normal subgroup $\mathcal{R}$ generated by elements of the form $r_1^{-1}\cdot r_2\cdot r_1\cdot (r_2^{r_1})^{-1}$, where $r_i\in F_i\in \mathcal{F}$ and $r_2^{r_1}$ denotes the conjugation in $G$. We have an epimorphism $\tau : \amalg_{F\in \mathcal{F}} F\rightarrow S$.

Let $f: X\rightarrow Y$ be a $C_n$-equivalence and $h\in \mathrm{Hom}(S,\, Y)$. Composition with $\tau$ gives a homomorphism $h\tau:\amalg_{F\in \mathcal{F}} F\rightarrow Y$. By Proposition 7.1 in \cite{chac}, the group $\amalg_{F\in \mathcal{F}} F$ is $C_n$-cellular. Hence there exists a homomorphism $t':\amalg_{F\in \mathcal{F}} F\rightarrow X$ such that $ft'=h\tau$. This implies $ft'(r_1^{-1}\cdot r_2\cdot r_1\cdot (r_2^{r_1})^{-1})=1$, that is $ft'(r_1^{-1}\cdot r_2\cdot r_1)=ft'(r_2^{r_1})$. Now, $t'(r_1^{-1}\cdot r_2\cdot r_1)$ and $t'(r_2^{r_1})$ are both in $X_n$, so by the injectivity of $f$ on this set, we have $t'(r_1^{-1}\cdot r_2\cdot r_1\cdot (r_2^{r_1})^{-1})=1$. This means that $t'$ can be extended to $t:S\rightarrow X$ and we have $ft'=ft\tau=h\tau$. But $\tau$ is a surjective homomorphism so $ft=h$.

Now suppose $t_1$ and $t_2$ are two homomorphisms from $S$ to $X$ such that $ft_1=ft_2$. Clearly, this gives $ft_1\tau=ft_2\tau$. Since $\amalg_{F\in \mathcal{F}} F$ is $C_n$-cellular this implies $t_1\tau=t_2\tau$, and we have $t_1=t_2$.
\end{proof}



\begin{coro}
Suppose $G$ is a finite group. Let $n$ be a positive integer and $\mathcal{F}$ be a family of subgroups of $G$ satisfying the hypotheses in Theorem \ref{sumcel}. If $S$ is the active sum of $\mathcal{F}$, then $\pi(H_2(S))\subseteq \pi(n)$.
\end{coro}
\begin{proof}
By the previous theorem, $S$ is $C_n$-cellular. This implies, by Corollary 4 in \cite{chac}, that $H_2(S)$ is $C_n$-constructible. But, using Proposition 4.3.1 of the same reference, one can show that if $A$ and $K$ are finite nilpotent groups, then $K$ is $A$-constructible if and only if $\pi(K)\subseteq \pi(A)$. This gives us the result.
\end{proof}

\begin{coro}
Suppose $G$ is a finite group. Let $\mathcal{F}$ be a generating family of subgroups of $G$ satisfying the hypotheses in Theorem \ref{sumcel}. Let $\varphi: S\rightarrow G$ be the canonical  surjective homomorphism from the active sum $S$ of $\mathcal{F}$ to $G$. If $\pi (H_2(G))\subseteq \pi (n)$  and $\varphi$ is a $C_n$-injection, then $\varphi$ is an isomorphism from $S$ onto $G$.
\end{coro}
\begin{proof}
By the previous theorem, $S$ is $C_n$-cellular. Then, $S$ is $C_n$-generated, by Proposition 2.3 in \cite{chac}. The result follows now from  Corollary 5.4.3 in \cite{chac}. 

\end{proof}

\section{Examples}

As a consequence of Theorem \ref{sumcel}, we have the following two examples.

\begin{itemize}
\item Every Coxeter group is $C_2$-cellular.

By Example 2.2.4 in \cite{sumasact}, every Coxeter group is the active sum of a family of subgroups of order 2.		
		
\item Let $n\geq 3$. The group $\mathrm{SL}(n,\, q)$ is $C_3$-cellular if it is not one of the following: $\mathrm{SL}(3,\, 2)$, $\mathrm{SL}(3,\, 3)$, $\mathrm{SL}(4,\, 2)$ and $\mathrm{SL}(3,\, 4)$.

By Theorem 3.5 in \cite{sumasact}, each of these groups is the active sum of a family of subgroups of order 3.
\end{itemize}

\begin{rem}
These examples can also be obtained using Corollary 4 and Proposition 4.3.1 in \cite{chac}.
\end{rem}

\section*{Acknowledgements}
Thanks to J\'er\^ome Scherer for the fruitful conversations.
\begin{flushleft}
E-mail: \texttt{nadia.romero$\, $@$\, $ugto.mx}
\end{flushleft}

\end{document}